\documentclass{amsart}
\usepackage{mathrsfs,amssymb,multirow,setspace,color}
\usepackage{hyperref}
\usepackage[a4paper, total={7in, 9in}]{geometry}
\theoremstyle{plain}
\usepackage{graphicx}
\usepackage[displaymath, mathlines, pagewise]{lineno}

\newtheorem{theorem}{Theorem}[section]
\newtheorem{lemma}[theorem]{Lemma}
\newtheorem{proposition}[theorem]{Proposition}
\newtheorem{corollary}[theorem]{Corollary}

\theoremstyle{definition}

\newtheorem{example}[theorem]{Example}

\theoremstyle{remark}
\newtheorem{remark}[theorem]{Remark}

\input xy
\xyoption{all}
\def\a{\mathcal{P}_E(G)}

\def\d{\mathcal{D}(G)}

\begin{document}
\title{On the Difference Graph of power graphs of finite groups}



\author[Parveen, Jitender Kumar, ramesh prasad panda]{Parveen, Jitender Kumar$^{*}$, Ramesh prasad panda }
\address{$\text{}^1$Department of Mathematics, Birla Institute of Technology and Science Pilani, Pilani 333031, India}
\address{$\text{}^2$ Department of Mathematics, VIT-AP University, Amaravati, PIN-522237, Andhra Pradesh, India}
\email{p.parveenkumar144@gmail.com, jitenderarora09@gmail.com, rameshprpanda@gmail.com}

\begin{abstract}
The power graph of a finite group $G$ is a simple undirected graph with vertex set $G$ and two vertices are adjacent if one is a power of the other. The enhanced power graph of a finite group $G$ is a simple undirected graph whose vertex set is the group $G$ and two vertices $a$ and $b$ are adjacent if there exists $c \in G$ such that both $a$ and $b$ are powers of $c$.  In this paper, we investigate the difference graph $\mathcal{D}(G)$ of a finite group $G$, which is the difference of the enhanced power graph and the power graph of $G$ with all isolated vertices removed. We study the difference graphs of finite groups with forbidden subgraphs among other results. We first characterize an arbitrary finite group $G$ such that $\mathcal{D}(G)$ is a chordal graph, star graph, dominatable, threshold graph, and split graph. From this, we conclude that the latter four graph classes are equivalent for $\mathcal{D}(G)$. By applying these results, we classify the nilpotent groups $G$ such that $\mathcal{D}(G)$ belong to the aforementioned five graph classes. This shows that all these graph classes are equivalent for $\mathcal{D}(G)$ when $G$ is nilpotent. Then, we characterize the nilpotent groups whose difference graphs are cograph, bipartite, Eulerian, planar, and outerplanar. Finally, we consider the difference graph of non-nilpotent groups and determine the values of $n$ such that the difference graphs of the symmetric group $S_n$ and alternating group $A_n$ are cograph, chordal, split, and threshold. 
 \end{abstract}


\keywords{Enhanced power graph, power graph, nilpotent groups, forbidden subgraphs. \\ *  Corresponding author}

\maketitle
\section{Introduction}
There are a number of graphs attached to groups, \emph{e.g.}, Cayley graphs, commuting graphs, and conjugacy class graph. These graphs have been studied extensively in literature because of their numerous applications (see \cite{bertram1983,bianchi1992,a.hayat2019novel,a.kelarev2009cayley}). The \emph{power graph} $\mathcal{P}(G)$ of a finite group $G$ is a simple undirected graph with vertex set $G$ such that two vertices $a$ and $b$ are adjacent if one is a power of the other or equivalently: either $a \in  \langle b\rangle$ or $b \in \langle a\rangle$. Kelarev and Quinn \cite{a.kelarev2004combinatorial} introduced the concept of a directed power graph. Later on, the undirected power graphs of groups have been studied in various aspects, see \cite{SantiagoArguello, a.Cameron2011, a.Feng2015metricdimmension, kirkland2018} and references therein. Graph classes, such as cographs, chordal graphs, split graphs, and threshold graphs, can be defined in terms of forbidden induced subgraphs. Graphs with forbidden subgraphs appear in extremal graph theory and  in some other contexts. Certain forbidden subgraphs of undirected power graphs of groups have been studied in \cite{a.doostabadiforbidden, a.MannaForbidden2021}. For more results on power graphs of groups, we refer the reader to  \cite{Abawajy2013power,a.powergraphsurvey} and references therein. 

The \emph{commuting graph} $\Delta(G)$ of a finite group $G$ is a simple undirected graph with vertex set $G$ and two distinct vertices $x, y$ are adjacent if $xy = yx$. This graph has also been studied by taking $G {\setminus} Z(G)$ as the vertex set, where $Z(G)$ is the center of $G$. The commuting graph of a group is well-studied topic in literature, see \cite{bertram1983, a.britnell2013perfect, a.iranmanesh2008, a.kumar2021} and references therein.  Aalipour \emph{et al.} \cite{a.Cameron2016} characterized the finite group $G$ such that the power graph $\mathcal{P}(G)$ and the commuting graph $\Delta(G)$ are not equal, and hence they introduced a new graph between power graph and commuting graph called enhanced power graph. The \emph{enhanced power graph} $\a$ of a finite group $G$ is a simple undirected graph with vertex set $G$ and two vertices $x$ and $y$ are adjacent if $x,y\in \langle z \rangle$ for some $z\in G$. Equivalently, two vertices $x$ and $y$ are adjacent in $\a$ if and only if $\langle x,y\rangle$ is a cyclic subgroup of $G$. Aalipour \emph{et al.} \cite{a.Cameron2016} computed the clique number of enhanced power graph of an arbitrary group in terms of orders of its elements. Zahirovi$\acute{c}$ \emph{et al.} \cite{a.zahirovic2020study} proved that two finite abelian groups are isomorphic if their enhanced power graphs are isomorphic. Costanzo \emph{et al.} \cite{costanzo2021} investigated the connectedness and diameter of enhanced power graphs of finite groups with identity vertex removed. Other interesting results on graphs parameters and characterization problems involving enhanced power graphs are obtained in \cite{a.Bera2017, a.dalal2021enhanced, mahmoudifar2022, a.panda2021enhanced}.
For a detailed list of results and open problems on enhanced power graphs, we refer the reader to \cite{a.masurvey2022}. 

There is a hierarchy containing power graph, enhanced power graph, and commuting graph: one is a spanning subgraph of the next. Hence, it is natural to study the difference graph of these graphs. Let $G$ be a finite group. Aalipour \emph{et al.} {\rm \cite[Question 42]{a.Cameron2016}} motivated the researchers to study the connectedness of the difference graph $\Delta(G) - \mathcal{P}(G)$ of the commuting graph and power graph of $G$, i.e., the graph with vertex set $G$ in which $x$ and $y$ are adjacent if they commute, but neither is a power of the other. Further, Cameron \cite{a.camerongraphdefinedongroups2022} discussed some developments on the difference graph $\Delta(G) - \mathcal{P}(G)$. Moreover, some results were also given on the difference graph $\Delta(G)- \mathcal{P}_{E}(G)$ of commuting graph and enhanced power graph in \cite{a.camerongraphdefinedongroups2022}. Motivated by the above results, Biswas \emph{et al.} \cite{a.biswas2022difference} studied the difference graph $\d := \mathcal{P}_{E}(G) - \mathcal{P}(G)$ of enhanced power graph and power graph of $G$ with all isolated vertices removed. For certain group classes, they investigated the connectedness and perfectness of $\d$. In this paper, we study some interplay between algebraic properties of a group $G$ and graph theoretic properties of $\d$. Among other results, we characterize the finite groups $G$ such that $\d$ is a graph with some forbidden induced graph.

This paper is organized as follows. In Section 2, we present the necessary background materials and fix our notations. For a finite group $G$, in Section 3, we first investigate all possible dominating vertices of $\d$. We provide equivalent conditions on $G$ such that $\mathcal{D}(G)$ is a chordal graph, star graph, dominatable, threshold graph, and split graph. Subsequently, we conclude that the latter four graph classes are equivalent for $\mathcal{D}(G)$. In Section 4, we classify the finite nilpotent groups $G$ such that $\mathcal{D}(G)$ belong to the above five graph classes. In particular, we prove that all these graph classes are equivalent for $\mathcal{D}(G)$ when $G$ is nilpotent. Furthermore, we characterize the nilpotent groups whose difference graphs are cograph, bipartite, Eulerian, planar, and outerplanar. We study the difference graph of non-nilpotent groups in Section 5. We classify the values of $n$ for which the difference graphs of the symmetric group $S_n$ and alternating group $A_n$ are the aforementioned graphs with forbidden induced graphs.

\section{Preliminaries}
In this section, we state the necessary definitions and results. We also fix various notations. We denote $[r] = \{1,2,\dots, r\}$. Let $G$ be a group. We write $|G|$ to denote the order of $G$ and $o(x)$ for the order of an element $x$ in $G$. By $\langle x, y\rangle$, we mean the subgroup of $G$ generated by $x$ and $y$. Consider the notation $\pi _G=\{o(g): g \in G\}$. The \emph{exponent} of a group is defined as the least common multiple of the orders of all elements of the group.  For $m \geq 3$, the \emph{generalized quaternion group} $D_{2m}$ is a group of order $2m$ is defined in terms of generators and relations as 
$$D_{2m} = \langle x, y  :  x^{m} = y^2 = e,  xy = yx^{-1} \rangle.$$ 
For $n \geq 2$, the \emph{dicyclic group} $Q_{4n}$ is a group of order $4n$ is defined in terms of generators and relations as
$$Q_{4n} = \langle a, b  :  a^{2n}  = e, a^n= b^2, ab = ba^{-1} \rangle.$$
A cyclic subgroup of a group $G$ is called a \emph{maximal cyclic subgroup} if it is not properly contained in any cyclic subgroup of $G$ other than itself. Note that if $G$ is a cyclic group, then $G$ is  the only maximal cyclic subgroup of $G$. A finite group $G$ is said to be an \emph{EPPO-group} if the order of each element of $G$ is of prime power. The following results are useful in the sequel.

\begin{theorem}{\rm \cite{b.dummit1991abstract}}{\label{nilpotent}}
 Let $G$ be a finite group. Then the following statements are equivalent:
 \begin{enumerate}
     \item[(i)] $G$ is a nilpotent group.
     \item[(ii)] Every Sylow subgroup of $G$ is normal.
    \item[(iii)] $G$ is the direct product of its Sylow subgroups.
    \item[(iv)] For $x,y\in G, \  x$ and $y$ commute whenever $o(x)$ and $o(y)$ are relatively primes.
 \end{enumerate}
 \end{theorem}
 \begin{lemma}{\rm \cite[Lemma 2.5]{a.dalal2022lambda}}{\label{nilpotent lcm}}
     Let $G$ be a finite nilpotent group. Then for any $x,y \in G$, there exists $z\in G$ such that $o(z)=\mathrm{lcm}(o(x),o(y))$.
 \end{lemma}
 
Now,     we state the necessary graph theoretic definitions and fix notations for graphs. Throughout, we consider only \emph{simple graphs}, i.e., graphs with no loops and no parallel edges.  Consider a graph graph $\Gamma$ with the vertex set $V(\Gamma)$ and the edge set $E(\Gamma)$. If the vertices $v_1$ and $v_2$ are adjacent, we write $v_1 \sim v_2$. Otherwise, we write $v_1 \nsim v_2$. 
A \emph{walk} $\gamma$ in a graph $\Gamma$ from the vertex $u$ to the vertex $v$ is a sequence of vertices $u=u_1, u_2,\ldots , u_m = v(m >1)$ such that $u_i\sim u_{i+1}$ for every $i\in \{1,2,\ldots ,m-1\}$. If no edge is repeated in $\gamma$, then it is called a \emph{trail} in $\Gamma$. 
A \emph{path} of length $k$ between two vertices $u$ and $v$ is a sequence of $k+1$ distinct vertices starting from $u$ and ending with $v$ such that the consecutive vertices are adjacent.
The \emph{degree} $\mathrm{deg}(v)$ of a vertex $v$ in a graph $\Gamma$ is the number of vertices adjacent to $v$. A vertex $v$ of $\Gamma$ is said to be a \emph{dominating vertex} if $v$ is adjacent to all the other vertices of $\Gamma$. A subgraph $\Gamma '$ of $\Gamma$ is an \emph{induced subgraph} if two vertices of $V(\Gamma')$ are adjacent in $\Gamma'$ if and only if they are adjacent in $\Gamma$. 
The path on $n$ vertices and the cycle of length $n$ are denoted by $P_n$ and $C_n$, respectively. 
The complete graph on $n$ vertices is denoted by $K_n$. Also, we denote $2K_2$ by the disjoint union of two copies of a complete graph $K_2$. 

A graph $\Gamma$ is said to be \emph{bipartite} if $V(\Gamma)$ can be partitioned into two subsets such that no two vertices in the same subset of the partition are adjacent. 
A graph $\Gamma$ is \emph{Eulerian} if $\Gamma$ is connected and has a closed trail (walk with no repeated edge) containing all the edges of a graph. 
  A graph $\Gamma$ is \emph{planar} if it can be drawn on a plane without edge crossing. A planar graph is said to be \emph{outerplanar} if it can be drawn in the plane such that all its vertices lie on the outer face.
A \emph{complete bipartite} graph is a bipartite graph such that every vertex in one part is adjacent to all the vertices of the other part. A complete bipartite graph with partition size $m$ and $n$ is denoted by $K_{m, n}$. A complete bipartite graph $K_{1,n}$ is called a \emph{star graph}.  
 Now, we state the following results from \cite{b.westgraph} about some of these graphs. 
  
    \begin{theorem}{\label{bipartitecondition}}
A  graph $\Gamma$ is bipartite if and only if it has no odd cycle.
\end{theorem}

  \begin{theorem}
A connected graph $\Gamma$ is Eulerian if and only if $deg(v)$ is even for all $v\in V(\Gamma)$.
\end{theorem}

\begin{theorem}{\label{planarcondition1}}
A graph $\Gamma$ is planar if and only if it does not contain a subdivision of $K_5$ or $K_{3,3}$.
\end{theorem}

  \begin{theorem}{\label{outerplanarcondition}}
A graph $\Gamma$ is outerplanar if and only if it does not contain a subdivision of $K_4$ or $K_{2,3}$.
\end{theorem}

  In this paper, we investigate graph theoretic properties of the difference graph of a finite group using forbidden subgraphs. In this regard, we shall recall the following definitions.
  A graph is \emph{$\Gamma$-free} if it contains no induced subgraph isomorphic to $\Gamma$. 
  A graph $\Gamma$ is a \emph{cograph} if it is $P_4$-free. A graph $\Gamma$ is called \emph{split} if its vertex set is the disjoint union of two subsets $A$ and $B$ so that $A$ induces
a complete graph and $B$ induces an empty graph. It is well known that a graph $\Gamma$ is split if and only if $\Gamma$ is $C_4$-free, $C_5$-free, and $2K_2$-free. A graph $\Gamma$ is said to be \emph{chordal} if it is $C_n$-free for all $n\geq 4$. In other words, a chordal graph is a graph in which every cycle of length greater than $3$ has a chord. It is known that if a graph $\Gamma$ is $C_4$-free and $P_4$-free, then $\Gamma$ is chordal.
However, the converse need not be true. A graph $\Gamma$ is \emph{threshold} if $\Gamma$ has no induced subgraph isomorphic to $C_4$, $P_4$, or $2K_2$. The following results are useful in the sequel.

\begin{remark}{\label{order not divide}}
Let $x$ and $y$ be two elements of a finite group $G$ such that neither $o(x)\mid o(y)$ nor $o(y)\mid o(x)$. Then $x\nsim y$ in $\mathcal{P}(G)$. The converse is also true if $x$ and $y$ belong to the same cyclic subgroup of $G$.
\end{remark}  

\begin{proposition}{\rm \cite[Proposition 2.1]{a.biswas2022difference}}{\label{proposition 2.1}}
    Let $G$ be a group with order greater than $1$. Then a non-identity element $g\notin V(\d)$ if and only if either $\langle g\rangle$ is a maximal cyclic subgroup of $G$ or that every cyclic subgroup of $G$ containing $g$ has prime-power order.
\end{proposition}

\begin{proposition}{\rm \cite[Proposition 4.1]{a.biswas2022difference}}{\label{nilpotent coprime adj}}
In a finite nilpotent group $G$, if $x,y\in G$ such that $\mathrm{gcd}(o(x),o(y))=1$, then $x\sim y$ in $\d$.
\end{proposition}

Hereafter, given a finite group $G$, we simply refer $\d$ to as the \emph{difference graph} of $G$.

\begin{lemma}{\rm \cite[Lemma 4.4]{a.biswas2022difference}}{\label{induced lemma}}
Let $G$ be a group and $H$ a subgroup of $G$. If $H$ is not an EPPO-group itself, then $\mathcal{D}(H)$ is an induced subgraph of $\d$.
\end{lemma}

\section{Difference graph of a finite group}

In this section, we study the difference graph $\d$ of an arbitrary finite group $G$. In this connection, we first investigate the dominating vertices of $\d$. Then, we give a necessary and sufficient condition for a difference graph of a finite group to be chordal graph, star graph, dominatable, threshold graph, and split graph. We begin with the following lemma to study dominating vertices of $\d$.

\begin{lemma}{\label{xinverse}}
Let $G$ be a finite group which is not an EPPO-group. If $x\in V(\d)$, then $x^{-1} \in V(\d)$. Moreover, $x\nsim x^{-1}$ in $\d$.
\end{lemma}
\begin{proof}
Let $x\in V(\d)$. Then there exists a vertex $y\in G$ such that $x\sim y$ in $\d$. It follows that $x,y \in \langle z\rangle$ for some $z\in G$ and $x\notin \langle y\rangle$, $y \notin \langle x \rangle $. Now, $x\in \langle z \rangle$ implies that $x^{-1} \in \langle z \rangle $, and $\langle x \rangle =\langle x^{-1}\rangle$ implies that $x^{-1}\notin \langle y \rangle , y\notin \langle x^{-1} \rangle $. Consequently, $x^{-1}\sim y$ in $\d$ and so $x^{-1}\in V(\d)$. Furthermore, $x\sim x^{-1}$ in $\a$ and $\mathcal{P}(G)$. It follows that $x\nsim x^{-1}$ in $\d$.
\end{proof}

\begin{lemma}{\label{order2}}
Let $G$ be a finite group such that $G$ is not an EPPO-group. If $x\in V(\d)$ is a dominatable vertex, then $o(x)=2$.
\end{lemma}

\begin{proof}
Let $x\in G$ be a dominating vertex in $\d$. By Lemma \ref{xinverse}, $x\nsim x^{-1}$ in $\d$. Consequently, either $x=e$ or $o(x)=2$. But $e\notin V(\d)$ and so $o(x)=2$.
\end{proof}

\begin{lemma}{\label{one dominating}}
Let $G$ be a finite group which is not an EPPO-group. Then $\d$ can have at most one dominating vertex. In particular, $\d$ is not a complete graph.
\end{lemma}

\begin{proof}
On the contrary, assume that $x$ and $y$ are two dominating vertices of $\d$. By Lemma \ref{order2}, $o(x)=o(y)=2$. Since $x\sim y$ in $\d$, there exists $z\in G$ such that $x,y \in \langle z \rangle$, which is a contradiction to the fact that a cyclic group can contain at most one element of order $2$. Thus, $\d$ has at most one dominating vertex. Since $\d$ has more than one vertices, this implies that it is not complete.
\end{proof}


\begin{proposition}{\label{order divide}}
Let be a finite group $G$ and $x$, $y$ be two non-identity elements of $G$ such that $o(x)\mid o(y)$. Then $x$ is not adjacent to $y$ in $\d$.
\end{proposition}

\begin{proof}
If $x\nsim y$ in $\a$, the result holds trivially. We may now assume that $x\sim y$ in $\a$. It implies that $x,y\in \langle z \rangle$ for some $z\in G$. Now to prove our result, it is sufficient to show that $\langle x\rangle \subseteq \langle y \rangle$, so that $x \sim y$ in $\mathcal{P}(G)$. Let $o(x)=m$. Since $o(x)\mid o(y)$, we have a cyclic subgroup $H$ of order $m$ contained in $\langle y \rangle$. Consequently, $H$ is a subgroup of $\langle z \rangle$. Also, $\langle x \rangle$ is a cyclic subgroup of order $m$ contained in $\langle z\rangle$. Since a cyclic group can contain at most one cyclic subgroup of a particular order, we obtain $\langle x\rangle =H$ and so $\langle x\rangle \subseteq \langle y \rangle$.
\end{proof}

\begin{corollary}{\label{x nonadj y in P}}
Let $G$ be a finite nilpotent group. If $x$ are $y$ are two non-identity elements of a Sylow subgroup of $G$, then $x\nsim y$ in $\d$.
\end{corollary}
Note that a maximum order element of a finite group generates a maximal cyclic subgroup. Taking Proposition \ref{proposition 2.1} into consideration, we have the following remark.
\begin{remark}{\label{maximum order}}
If $x$ is an element of maximum possible order in a finite group $G$, then $x\notin V(\d)$.
\end{remark}

For a finite group $G$, we define
\begin{itemize}
    \item \emph{$G$ satisfies condition $\mathcal{A}$} if $\pi _G\subseteq \{1,2p_1,2p_2,\ldots ,2p_k\}\cup \{p^{\alpha} : \alpha \in \mathbb{N} \}$, where $p$ and $p_i >2(1\leq i \leq k)$ are primes, and the cardinality of a subgroup, which is the intersection of any two cyclic subgroups of the order $2p_i$ ($1\le i  \le k$), is at most $2$.
    
    \item \emph{$G$ satisfies condition $\mathcal{B}$} if $\pi _G\subseteq \{1,2p_1,2p_2,\ldots ,2p_k\}\cup\{p^{\alpha} : \alpha \in \mathbb{N} \}$, where $p$ and $p_i >2(1\leq i \leq k)$ are primes, and the cardinality of a subgroup, which is the intersection of any two cyclic subgroups of the orders $2p_i$ ($1\le i \le k$) and $2p_j$ ($1\le j  \le k$), is $2$.
\end{itemize}
 For a finite group $G$, note that if $G$ satisfies condition $\mathcal{B}$, then $G$ satisfies condition $\mathcal{A}$.

\begin{proposition}{\label{c4free arbitrary}}
Let $G$ be a finite group which is not an EPPO-group. Then $\d$ is $C_4$-free if and only if $G$ satisfies condition $\mathcal{A}$.
\end{proposition}
\begin{proof}
First, suppose that $\d$ is a $C_4$-free graph. If possible, let $G$  contain an element $x$ of order $pq$, where $p,q$ are distinct odd primes. Then there exist $x_1,x_2, y_1,y_2 \in \langle x\rangle$ such that $o(x_1)=o(x_2)=p$ and $ o(y_1)=o(y_2)=q$. Note that the subgraph induced by these four elements is isomorphic to a complete graph $K_4$ in $\a$. By Remark \ref{order not divide} and Proposition \ref{order divide}, we get an induced cycle $x_1\sim y_1\sim x_2 \sim y_2\sim x_1$ isomorphic to $C_4$ in $\d$, a contradiction. \\
Now suppose $G$ has an element $x$ of the order $2p^{\alpha}$ for some odd prime $p$ and $\alpha \geq 2$. Then there exist $x_1,x_2,y_1,y_2 \in \langle x\rangle$ such that $o(x_1)=o(x_2)=p^2$ and $o(y_1)=o(y_2)=2p$. Clearly, the subgraph induced by the set $\{x_1,x_2,y_1,y_2\}$ is a complete graph $K_4$ in $\a$. By Proposition \ref{order divide} and Remark \ref{order not divide}, the subgraph induced by the set $\{x_1,x_2,y_1,y_2\}$ is isomorphic to $C_4$ in $\d$,  a contradiction.\\
Now assume that $G$ has two cyclic subgroups $M$ and $N$ of order $2p$ for some odd prime $p$ such that $|M\cap N|=p$. Then there exists $z_1,z_2 \in M\cap N$ such that $o(z_1)=o(z_2)=p$. Let $x$ and $y$ be the elements of order $2$ in $M$ and $N$, respectively. Notice that $x\sim z_1$, $x\sim z_2$, $y\sim z_1$ and $y\sim z_2$ in $\a$. By Remark \ref{order not divide} and Proposition \ref{order divide}, we obtain an induced cycle $x\sim z_1\sim y \sim z_2\sim x$ isomorphic to $C_4$ in $\d$, again a contradiction. Consequently, $G$ satisfies condition $\mathcal{A}$.

Conversely, assume that $G$ satisfies condition $\mathcal{A}$. To prove $\d$ is a $C_4$-free, on the contrary, we assume that $G$ has an induced cycle $x\sim y \sim z \sim t \sim x$. Notice that either $o(x)=2$ or $o(x)=p_i$ for some $i\in [k]$, otherwise $x\notin V(\d)$, see Proposition \ref{proposition 2.1}. First let $o(x)=2$. Since $x\sim y$ in $\d$, we get $x,y\in \langle g\rangle$ for some $g\in G$ and neither $o(y)\mid o(x)$ nor $o(x) {\mid} o(y)$. It follows that $|\langle g \rangle|=2p_i$ and $o(y)=p_i$ for some $i\in [k]$. Since $y\sim z$ in $\d$, we obtain $y,z\in \langle g' \rangle$ for some $g'\in G$ and neither $o(y)\mid o(z)$ nor $o(z)\mid o(y)$. Consequently, $|\langle g' \rangle|=2p_i$ and $o(z)=2$. Thus, $y\in \langle g\rangle \cap \langle g' \rangle$. It follows that $|\langle g\rangle \cap \langle g' \rangle|=p_i$, which is not possible. \\
We may now assume that $o(x)=p_i$ for some $i\in [k]$. Since $x\sim y$ in $\d$, we obtain $x,y\in \langle l\rangle$ for some $l\in G $ and neither $o(y)\mid o(x)$ nor $o(x)\mid o(y)$. Consequently, $|\langle l \rangle|=2p_i$ and $o(y)=2$. Similarly, $x\sim t$ in $\d$ implies that $x,t\in \langle l'\rangle$ for some $l'\in G$ and neither $o(t)\mid o(x)$ nor $o(x)\mid o(t)$. It follows that $|\langle l' \rangle|=2p_i$ and $o(t)=2$. Thus, $x\in \langle l\rangle \cap \langle l' \rangle$ and so $|\langle l\rangle \cap \langle l' \rangle|=p_i$, again a contradiction. Hence, $\d$ is a $C_4$- free graph.
\end{proof}

\begin{theorem}{\label{chordal arbitrary}}
Let $G$ be a finite group which is not an EPPO-group. Then $\d$ is a chordal graph if and only if $G$ satisfies condition $\mathcal{A}$.
\end{theorem}

\begin{proof}
In view of Proposition \ref{c4free arbitrary}, we prove the desired result by showing that the graph $\d$ does not contain an induced cycle of length greater than $4$. Notice that the vertex set of $\d$ contains only elements of the orders $2, p_1, p_2,\ldots , p_k$ (cf. Proposition \ref{proposition 2.1}). We claim that $\d$ is a bipartite graph. Consider a partition  of $V(\d)$ into two subsets $A$ and $B$ such that if $o(x)=2$ then $x\in A$, otherwise $x\in B$. By Proposition \ref{order divide}, no two elements of $A$ are adjacent in $\d$. For $a,b \in B$ such that $o(a)=o(b)$, we have $a\nsim b$ in $\d$. Now, let $x,y \in B$ such that $o(x)=p_i$ and $o(y)=p_j$ for distinct $i,j\in [k]$. Then $x\nsim y$ in $\d$ because $G$ does not contain an element of the order $p_ip_j$. Thus, $\d$ is a bipartite graph. Since the cardinality of the intersection of any two distinct cyclic subgroups of the order $2p_i$ is at most $2$, we obtain that the elements of order $p_i$ ($1\leq i\leq k$) in $B$ is adjacent exactly one element of order $2$ in $\d$. Thus the length of any path in $\d$ is at most $2$. This completes the proof.
\end{proof}

\begin{corollary}
Let $G$ be a finite group which is not an EPPO-group. Then $\d$ is a chordal graph if and only if it is $C_4$-free.
\end{corollary}

In the following, we characterize finite groups whose difference graphs are star graphs and dominatable. 

\begin{theorem}{\label{star arbitrary}}
    Let $G$ be a finite group which is not an EPPO-group. Then the following conditions are equivalent:
   \begin{itemize}
    \item[(i)] $\d$ is a star graph. 
    \item[(ii)] $\d$ is dominatable.
    \item[(iii)] $G$ satisfies condition $\mathcal{B}$.
\end{itemize}
\end{theorem}
\begin{proof}
      Clearly, (i) $\implies $(ii). Now we show that (ii) $\implies $(iii) and (iii) $\implies $(i).\\
       (ii) $\implies $(iii) Let $\d$ be a dominatable graph. By Lemmas \ref{order2} and \ref{one dominating}, there exists an element $x$ of order $2$ such that $x\sim y$ for every $y\neq x \in V(\d)$. If possible, let $g\in G$ such that $o(g)=p_1p_2$ for some odd primes $p_1$ and $p_2$. There exists $y,z\in \langle g \rangle$ such that $o(y)=p_1$ and $o(z)=p_2$. Then $y \sim  z$ in $\a$ and by Remark \ref{order not divide}, $y\nsim z$ in $\mathcal{P}(G)$. It follows that $y\sim z$ in $\d$. Consequently, $y,z\in V(\d)$. Since $x$ is a dominating vertex in $\d$, we get $x\sim y$ and $x\sim z$. It implies that $\langle x,y\rangle$ and $\langle x,z\rangle$ are cyclic subgroups of $G$ generated by $xy$ and $xz$, respectively. Since $y$ commutes with $x$ and $z$, we obtain a cyclic subgroup $\langle xyz\rangle$ of $G$ containing $y$ and $xz$. Note that neither $o(y)\mid o(xz)$ nor $o(xz)\mid o(y)$. Thus $y\sim xz$ in $\d$ and so $xz\in V(\d)$. By Proposition \ref{order divide}, $x\nsim xz$ in $\d$, a contradiction. Thus $G$ has no element of the order $p_1p_2$. \\
       Suppose $G$ has an element $y$ of order $2p^2$, for some odd prime $p$. Let $z,w\in \langle y \rangle$ be elements of the orders $p$ and $p^2$, respectively. We claim that $x\in \langle y \rangle$. On contrary, assume that $t\neq x\in \langle y \rangle$ such that $o(t)=2$. Clearly, $z\sim t$ in $\a$ and neither $o(t)\mid o(z)$ nor $o(z)\mid o(t)$. Thus $z \sim t$ in $\d$ and so $t\in V(\d)$. But by Proposition \ref{order divide}, $x\nsim t$ in $\d$, a contradiction. Thus, $x\in \langle y \rangle$. Notice that $xz\in \langle y \rangle$ is an element of the order $2p$ and $w \sim xz$ in $\a$. Note that neither $o(w)\mid o(xz)$ nor $o(xz)\mid o(w)$. Thus, $w \sim xz$ in $\d$ and so $xz\in V(\d)$. By Proposition \ref{order divide}, $x\nsim xz$ in $\d$, again a contradiction. Consequently, $G$ has no element of the order $2p^2$ for an odd prime $p$. \\
       Let $M$ and $N$ be two cyclic subgroups of the order $2p_i$ and $2p_j$, respectively, for some $i,j\in [k]$. To prove the result, it is sufficient to show that $x\in M\cap N$. If possible, let $x\notin M$. Suppose $y,z \in M$ such that $o(y)=2$ and $o(z)=p_i$. Then $y\sim z$ in $\a$ and neither $o(y)\mid o(z)$ nor $o(z)\mid o(y)$. Consequently, $y\sim z$ in $\d$ and so $y\in V(\d)$. By Proposition \ref{order divide}, $x\nsim y$ in $\d$, a contradiction. Thus $x\in M$. Similarly, we obtain $x\in N$. It follows that $|M\cap N|=2$. Thus $G$ satisfies condition $\mathcal{B}$.\\
       (iii) $\implies $(i) Let $G$ satisfy condition $\mathcal{B}$. Observe that if $x\in V(\d)$, then $x$ belongs to a cyclic subgroup of the order $2p_i$ for some $i\in [k]$ (see Proposition \ref{proposition 2.1}). By the proof of the Theorem \ref{chordal arbitrary}, notice that $\d$ is a bipartite graph and the partition set $A$ contains elements of the order $2$ and the partition set $B$ contains the elements of the orders $p_1,p_2,\ldots ,p_k$. First note that $|A|=1$. Assume that $x,y \in A$. Then there exist two cyclic subgroups $M$ and $N$ of the order $2p_i$ and $2p_j$, respectively, such that $x\in M, y\in N$. Since the group $G$ satisfies condition $\mathcal{B}$, we obtain $|M\cap N|=2$. Consequently, $M\cap N$ contains an element of order $2$. It follows that $x=y$. Thus, the partition set $A$ of $\d$ contains exactly one element, say $x$. Since $V(\d)=A\cup B$, we obtain that if $y\in B$ then $y\sim x$ in $\d$. Thus, $\d$ is a star graph with $x$ as the dominating vertex of $\d$.
\end{proof}

Next, we characterize finite finite groups whose difference graphs are threshold graphs and split graphs. 

\begin{theorem}{\label{threshold}}
    Let $G$ be a finite group which is not an EPPO-group. Then the following conditions are equivalent:
   \begin{itemize}
    \item[(i)] $\d$ is a threshold graph. 
    \item[(ii)] $\d$ is a split graph.
    \item[(iii)] $G$ satisfies condition $\mathcal{B}$.
\end{itemize}
\end{theorem}

\begin{proof}
  Clearly, (i) $\implies $(ii). By Theorem \ref{star arbitrary}, we have (iii) $\implies $(i).\\
   (ii) $\implies $(iii) Let $\d$ be a split graph. Then by Proposition \ref{c4free arbitrary}, $G$ satisfies condition $\mathcal{A}$. Now to prove the result, we show that for $1\le i,j  \le k$, the cardinality of the intersection of any two cyclic subgroups of the order $2p_i$ and $2p_j$, respectively, is $2$. On the contrary, assume there exist two cyclic subgroups $M$ and $N$ of the order $2p_i$ and $2p_j$, respectively, such that $|M\cap N|=1$. Suppose $x\in M$ and $y\in N$ are the elements of order $2$. Let $z\in M$ and $t\in N$ such that  $o(z)=p_i$ and $o(t)=p_j$. Then $x\sim z$ and $y\sim t$ in $\a$. By Remark \ref{order not divide}, $x\nsim z$ and $y\nsim t$ in $\mathcal{P}(G)$. It follows that  $x\sim z$ and $y\sim t$ in $\d$. Note that $G$ has no element of the order $p_ip_j$ as $G$ satisfies condition $\mathcal{A}$,  Consequently, $z\nsim t$ in $\a$ and so $z\nsim t$ in $\d$. Since $G$ satisfies condition $\mathcal{A}$, we get $x\nsim t$, $y\nsim z$ in $\a$. By Proposition \ref{order divide}, $x\nsim y$ in $\d$. Thus the subgraph induced by the set $\{x,y,z,t\}$ is isomorphic to $2K_2$ in $\d$, a contradiction. Thus $G$ satisfies condition $\mathcal{B}$.\\
 \end{proof}

In view of the above two theorems, we deduce that all four classes addressed in them are equivalent for difference graphs of finite groups. 
 
\section{Difference graph of nilpotent groups}

Let $G$ be a finite nilpotent group and $|G| = p_1^{\alpha_1} p_2^{\alpha_2} \dots p_r^{\alpha_r}$, where $p_1 < p_2 < \dots < p_r$ are primes and $\alpha_1, \alpha_2, \dots, \alpha_r$ are positive integers. We denote the Sylow $p_i$-subgroup of $G$ by $P_i$, for $i \in [r]$. Then, in view of Theorem \ref{nilpotent}, we have $G=P_1P_2\cdots P_r$ and that $|P_i|=p_i^{\alpha _i}$. Then for $x\in G$, there exists a unique element $x_i\in P_i$ for each $i\in [r]$ such that $x=x_1x_2\cdots x_r$. Since the mapping $(x_1,x_2,\ldots , x_r) \longmapsto x_1x_2\cdots x_r$ is a group isomorphism from $P_1\times P_2\times \cdots \times P_r$ to $P_1P_2\cdots P_r$, we sometimes write $P_1\times P_2\times \cdots \times P_r$ instead of $P_1P_2\cdots P_r$. Throughout this section, we use these notations for $G$ without mentioning explicitly.

In this section, classify the finite nilpotent groups $G$ such that the difference graph $\d$ is a chordal graph, star graph, dominatable, threshold graph, and split graph. Then we characterize the nilpotent groups whose difference graphs are cograph, bipartite, Eulerian, planar, and outerplanar.  
Note that a finite nilpotent group $G$ is an EPPO-group if and only if $G$ is a $p$-group.

\begin{theorem}{\label{example nilpotent star}}
Let $G$ be a finite nilpotent group which is not a $p$-group. Then the following conditions are equivalent:
\begin{itemize}
    \item[(i)] $G$ satisfies condition $\mathcal{A}$. 
    \item[(ii)] $G$ satisfies condition $\mathcal{B}$.
    \item[(iii)] $G \cong \mathbb{Z}_2 \times P$, where $P$ is a $p$-group of exponent $p> 2$.
\end{itemize}
\end{theorem}
\begin{proof}
(i) $\implies $(ii) Suppose that $G$ satisfies condition $\mathcal{A}$. Then we claim that $k=1$. On the contrary, assume that $G$ has elements of the order $2p_1$ and $2p_2$ for some odd primes $p_1$ and $p_2$. Since $G$ is a nilpotent group and so $G$ has elements of the order $2p_1p_2$ (cf. Lemma \ref{nilpotent lcm}), a contradiction. Thus $k=1$. If possible, assume that $M$ and $N$ are two cyclic subgroups  of $G$ of the order $2p_1$ such that $|M\cap N|=1$. Let $x\in M$ and $y\in N$ be elements of the order $2$. Suppose $z\in M$ such that $o(z)=p_1$. Observe that $o(yz)=2p_1$ and $|M\cap \langle yz\rangle |=p_1$, a contradiction. Thus $|M\cap N|=2$ and so $G$ satisfies condition $\mathcal{B}$.\\
(ii) $\implies $(iii) Let $G=P_1P_2\cdots P_r$ be a nilpotent group such that $|P_i|=p_i^{\alpha _i}$ and $p_j<p_{j+1}$ for $j\in [r-1]$. If possible, let $r\geq 3$. Then there exist $x,y,z\in G$ such that $o(x)=p_1$, $o(y)=p_2$ and $o(z)=p_3$. It follows that $o(yz)=p_2p_3$, which is not possible, as $G$ satisfies condition $\mathcal{B}$. Thus $r=2$. Consequently, we get $p_1=2$. Now, to prove the result, it is sufficient to show that $|P_1|=2$ and $\mathrm{exp}(P_2)=p_2$. First we show that $\mathrm{exp}(P_1)=2$. If possible assume that $\mathrm{exp}(P_1)=2^{\alpha}$ for $\alpha \geq 2$. Then there exists $x\in P_1$ such that $o(x)=2^{\alpha}$. Let $y\in P_2$ such that $o(y)=p_2$. Consequently, $o(xy)=2^{\alpha}p_2$, a contradiction. Thus $\mathrm{exp}(P_1)=2$. Similarly, we obtain $\mathrm{exp}(P_2)=p_2$. Now we show that $|P_1|=2$. On the contrary, assume that $P_1$ has two non-identity elements $x$ and $y$. Let $z\in P_2$ such that $o(z)=p_2$. Then $\langle xz \rangle$ and $\langle yz\rangle$ are two cyclic subgroups of order $2p_2$. Since $z\in \langle xz \rangle \cap \langle yz \rangle$, we obtain $|\langle xz \rangle \cap \langle yz \rangle|=p_2$, a contradiction. This completes the proof.\\
(iii) $\implies $(i) Let $G \cong \mathbb{Z}_2 \times P$, where $P$ is a $p$-group of exponent $p> 2$. Then $\pi _G=\{1,2,p,2p\}$ and the cardinality of the intersection of any two cyclic subgroups of the order $2p$ is $2$. Thus $G$ satisfies condition $\mathcal{A}$.

\end{proof}

In view of Theorems \ref{chordal arbitrary}, \ref{star arbitrary}, and \ref{example nilpotent star}, we have the following theorem. This theorem shows that all the five graph classes addressed in it are equivalent for $\d$ when $G$ is a finite nilpotent group, but not a $p$-group.

\begin{theorem}{\label{star condition}}
Let $G$ be a finite nilpotent group which is not a $p$-group. Then the following conditions are equivalent:
\begin{itemize}
    \item[(i)] $\d$ is a chordal graph.
    \item[(ii)] $\d$ is a star graph. 
    \item[(iii)] $\d$ is dominatable.
     \item[(iv)] $\d$ is a threshold graph. 
    \item[(v)] $\d$ is a spilt graph.
    \item[(vii)] $G \cong \mathbb{Z}_2 \times P$, where $P$ is a $p$-group of exponent $p> 2$.
\end{itemize}
\end{theorem}

We next investigate the difference graphs of finite nilpotent groups for cograph. Followed by this, we classify the finite nilpotent groups whose difference graphs are bipartite, Eulerian, planar, and outerplanar.

\begin{theorem}
Let $G$ be a finite nilpotent group which is not a $p$-group. Then $\d$ is a cograph if and only if $G\cong P_1 \times P_2$, where each $P_i$ is a $p_i$-group of exponent $p_i$ for every $i\in\{1,2\}$.
\end{theorem}

\begin{proof}
Let us assume that $\d$ is a cograph. Let $G=P_1P_2\cdots P_r$ such that $|P_i|=p_i^{\alpha _i}$ and $p_j<p_{j+1}$ for $j\in [r-1]$. If possible, let $r\geq 3$. Let $x,y,z$ be elements of order $p_1, p_2$ and $p_3$, respectively. Notice that $o(xz)=p_1p_3$ and $o(xy)=p_1p_2$. Now by Proposition \ref{nilpotent coprime adj}, $x\sim y$ and $y\sim xz$ in $\d$. Note that $x,y,z \in \langle xyz \rangle$. It follows that $xz, xy \in \langle xyz\rangle$. Note that neither $o(xy)\mid o(xz)$ nor $o(xz)\mid o(xy)$. Consequently, $xz\sim xy$ in $\d$. Since $\langle x \rangle\subseteq \langle xy \rangle$, $\langle x \rangle \subseteq \langle xz \rangle$ and $\langle y\rangle \subseteq \langle xy \rangle$, we have $x\nsim xy$, $x \nsim xz$ and $y \nsim xy$ in $\d$. Thus the subgraph induced by the set $\{x,y,xz,xy\}$ is isomorphic to the path graph $P_4$,  a contradiction. Thus $r=2$. 

Now, we have to show that the exponents of $P_1$ and $P_2$ are $p_1$ and $p_2$, respectively. We prove it for $P_1$ and similar result holds for $P_2$ as well. If possible, let $\mathrm{exp}(P_1)\geq p_1^2$. Then there exists an element $x$ such that $o(x)=p_1^2$. Let $y\in \langle x \rangle$ such that $o(y)=p_1$. Let $z\in G$ such that $o(z)=p_2$. Then by Proposition \ref{nilpotent coprime adj}, $y\sim z$ and $z\sim x$ in $\d$. Now $x,z\in \langle xz\rangle$. It follows that $x,yz \in \langle xz \rangle$ and neither $o(x)\mid o(yz)$ nor $o(yz)\mid o(x)$. Consequently, $x\sim yz$ in $\d$. Also, $o(y)$ and $o(z)$ divides $o(yz)$. Then by Proposition \ref{order divide}, $y\nsim yz$ and $z\nsim yz$ in $\d$. Since $y\in \langle x \rangle$, we obtain $x\nsim y$ in $\d$. Thus, the graph induced by the set $\{y,z,x,yz\}$ is isomorphic to the path graph $P_4$.

Conversely, let us assume that $G \cong P_1\times P_2$ such that $\mathrm{exp}(P_1)=p_1$ and $\mathrm{exp}(P_2)=p_2$. It follows that $G$ has elements of order $1$, $p_1$, $p_2$, and $p_1p_2$. Observe that $e\notin V(\d)$. By Remark \ref{maximum order}, the elements of order $p_1p_2$ does not belong to the vertex set of $\d$. By Proposition \ref{nilpotent coprime adj} and Corollary \ref{x nonadj y in P},  $\d$ is a complete bipartite graph $K_{p_1^{\alpha _1}-1,p_2^{\alpha _2}-1}$ and so $\d$ is a cograph. Thus, the result holds. 
\end{proof}

\begin{corollary}[{\rm \cite[Propostion 10.3]{a.biswas2022difference}}]
Let $G$ be a finite cyclic group of order $n$, which is not a $p$-group. Then $\d$ is a cograph if and only if $n$ is a product of two distinct primes.
\end{corollary}

\begin{theorem}{\label{bipartite nilpotent}}
Let $G$ be a finite  nilpotent group which is not a $p$-group. Then $\d$ is a bipartite graph if and only if $G\cong P_1\times P_2$, where $P_i$ is the Sylow $p_i$-subgroup of $G$ such that $\mathrm{exp}(P_i)=p_i$ for $i\in \{1,2\}$.
\end{theorem}
\begin{proof}
Let $G=P_1P_2\cdots P_r$ be a nilpotent group such that $|P_i|=p_i^{\alpha _i}$ and $p_j<p_{j+1}$ for $j\in [r-1]$. First assume that $\d$ is a bipartite graph. If possible, let $r\geq 3$. Let $x$, $y$, $z$ be elements of $G$ such that $o(x)=p_1$, $o(y)=p_2$ and $o(z)=p_3$. By Proposition \ref{nilpotent coprime adj}, $x\sim y$, $y\sim z$ and $z\sim x$ in $\d$. Thus we get a cycle of length $3$ in $\d$, a contradiction. It follows that $r=2$. Now suppose that $\mathrm{exp}(P_1)\geq p_1^2$ and $\mathrm{exp}(P_2)\geq p_2^2$. Then there exist $x,y \in G$ such that $o(x)=p_1^2$ and $o(y)=p_2^2$. Consider $x_1\in \langle x\rangle ,y_1\in \langle y\rangle$ such that $o(x_1)=p_1$ and $o(y_1)=p_2$. Then $o(x_1y_1)=p_1p_2$. Notice that $x,y \in \langle xy\rangle$ and so $x_1y_1 \in \langle xy \rangle$. It follows that $x\sim x_1y_1\sim y\sim x$ in $\a$. Note that neither $o(x_1y_1)\mid o(x)$ nor $o(x)\mid o(x_1y_1)$. By Remark \ref{order not divide}, $x\nsim x_1y_1$ in $\mathcal{P}(G)$. Similarly, $y\nsim x_1y_1$ and $x\nsim y$ in $\mathcal{P}(G)$. It follows that $x\sim x_1y_1 \sim y \sim x$ in $\d$, again a contradiction. \\
Conversely, let $G\cong P_1\times P_2$ such that either $\mathrm{exp}(P_1)=p_1$ or $\mathrm{exp}(P_2)=p_2$. Without loss of generality assume that $\mathrm{exp}(P_1)=p_1$ and $\mathrm{exp}(P_2)=p_2^k$ for some $k\geq 1$. Then $\pi_G=\{1,p_1,p_2,p_2^2,\ldots ,p_2^k,p_1p_2, p_1p_2^2,\ldots ,p_1p_2^k\} $. Consider a partition of $V(\d)$ into two sets $A$ and $B$ such that for $x\in V(\d)$, if $p_1\mid o(x)$ then $x\in A$, otherwise $x\in B$. If $x,y \in A$, then observe that either $o(x)\mid o(y)$ or $o(y)\mid o(x)$ and so $x\nsim y $ in $\d$ (see Proposition \ref{order divide}). Now if $a,b \in B$ then $o(a)=p^{\alpha}$ and $o(b)=p^{\beta}$. By Proposition \ref{order divide}, $a\nsim b$ in $\d$. Thus, $\d$ is a bipartite graph.
\end{proof}
\begin{corollary}
    Let $G$ be a finite cyclic group which is not a $p$-group. Then $\d$ is a bipartite graph if and only if $G \cong \mathbb{Z}_{p_1^{\alpha}p_2} $, for some distinct primes $p_1,p_2$ and $\alpha \geq 1$.
\end{corollary}

\begin{theorem}
Let $G$ be a finite nilpotent group which is not a $p$-group. Then $\d$ is Eulerian if and only if $|G|$ is odd.
\end{theorem}

\begin{proof}
Let $\d$ be an Eulerian graph. If possible, suppose $|G|$ is even. Then $G$ has an odd number of elements of order $2$. Let $x\in G$ be an element of order $p(\neq 2)$ for some prime $p$. Notice that if $x$ is adjacent to $y$, then $x$ is adjacent to $y^{-1}$. By Proposition \ref{nilpotent coprime adj}, $x$ is adjacent to every element of order $2$. Consequently, $\mathrm{deg}(x)$ is odd, a contradiction.

Conversely, suppose that $|G|$ is odd. Let $x$ be an arbitrary element of $V(\d)$. If $x\sim y$ then $x\sim y^{-1}$. It follows that the degree of $x$ in $\d$ is even. Thus the result holds.
\end{proof}

The following example shows that in the above theorem, the condition that $G$ is not a $p$-group is indeed necessary.

\begin{example} Consider $D_{30}=\langle x, y  :  x^{15} = y^2 = e,  xy = yx^{-1} \rangle$, is the dihedral group of order $30$. We observe that $\mathcal{D}(D_{30})\cong K_{2,4}$, see Figure \ref{fig2}. Hence, $\mathcal{D}(D_{30})$ is Eulerian, but $|D_{30}|$ is not odd.
\end{example}

\begin{figure}[ht]
    \centering
    \includegraphics[scale=0.9]{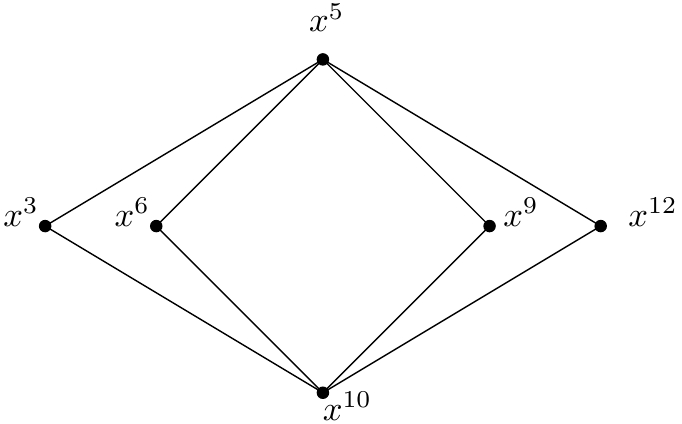}
    \caption{ The difference graph of $D_{30}$.}
    \label{fig2}
\end{figure}

\begin{theorem}{\label{PlanarD}}
Let $G$ be a finite nilpotent group which is not a $p$-group. Then $\d$ is planar if and only if $G$ is of one of the following forms:
\begin{itemize}
    \item[(i)] $G \cong \mathbb{Z}_2 \times P$, where $P$ is a $p$-group of exponent $p> 2$.
    \item[(ii)] $G \cong \mathbb{Z}_3 \times P$, where $P$ is a $p$-group of exponent $p> 3$.
    \item[(iii)] $G \cong P \times \mathbb{Z}_3$, where $P$ is a $2$-group with exponent $4$ and any two maximal cyclic subgroup of $P$ has trivial intersection.
    \item[(iv)] $G\cong \mathbb{Z}_2 \times \mathbb{Z}_2 \times \cdots \times \mathbb{Z}_2 \times \mathbb{Z}_3$.
\end{itemize}
\end{theorem}

\begin{proof}
Let $\d$ be a planar graph. Let $G=P_1P_2\cdots P_r$ such that $|P_i|=p_i^{\alpha _i}$ and $p_j<p_{j+1}$ for $j\in [r-1]$. If possible, let $r\geq 3$. Let $X=\{x_1, x_2, x_3\}$ such that $o(x_i)\in \{p_1,p_2\}$. Let $Y=\{y_1,y_2,y_3\}$ such that $o(y_i)=p_3$. By Proposition \ref{nilpotent coprime adj}, each element of $X$ is adjacent to every element of $Y$ in $\d$. Thus, the subgraph induced by the set $X\cup Y$ has a subgraph isomorphic to $K_{3,3}$, a contradiction. Thus, $r=2$.

Now, if possible, let $p_1\geq 5$. Then $G$ has at least $4$ elements of order $p_1$ and at least $6$ elements of order $p_2$. Let $X=\{x_1,x_2,x_3\}$ such that $o(x_i)=p_1$ and $Y=\{y_1,y_2,y_3\}$ such that $o(y_i)=p_2$. Then the subgraph induced by the set $X\cup Y$ is isomorphic to $K_{3,3}$; a contradiction. Consequently, $p_1\in \{2,3\}$.

Now we prove our result in the following two cases.

\noindent\textbf{Case-1:} $p_1=2$.
We discuss this case into the following four subcases.

\textbf{Subcase-1.1:} $\mathrm{exp}(P_1)\geq 8$. It follows that $G$ has at least one cyclic subgroup of order $8$. Let $x_1, x_2, x_3$ and $x_4$ be the generators of a cyclic subgroup of order $8$ in $G$. Let $y_1\in \langle x_1\rangle$ such that $o(y_1)=4$. Let $z_1, z_2$ be elements of order $p_2$. Now, $x_1, z_1\in \langle x_1 z_1\rangle$ implies that $y_1 z_1\in \langle x_1 z_1\rangle$. Note that neither $o(x_1)\mid o(y_1 z_1)$ nor $o(y_1 z_1)\mid o(x_1)$. Consequently, $y_1 z_1\sim x_1$ in $\d$. Similarly, $y_1 z_1\sim x_2$ and $y_1 z_1\sim x_3$ in $\d$. By Proposition \ref{nilpotent coprime adj}, $x_i\sim z_j$ for each $i\in \{1,2,3\}$, $j\in \{1,2\}$. Consequently the subgraph induced by the set $\{x_1,x_2, x_3, y_1z_1,z_1,z_2\}$ is isomorphic to $K_{3,3}$; a contradiction.

\textbf{Subcase-1.2:} $\mathrm{exp}(P_1)=4$. If $|P_2|> 3$, then the induced subgraph by the set $(P_1\cup P_2) \setminus \{e\}$ has a subgraph isomorphic to $K_{3,3}$ (see Proposition \ref{nilpotent coprime adj} and Corollary \ref{x nonadj y in P}). Thus, $|P_2|=3$ and so $P_2\cong \mathbb{Z}_3$. Now we show that the intersection of any two maximal cyclic subgroups of $P_1$ is trivial. On the contrary, assume that there exist two maximal cyclic subgroups $M_1$ and $M_2$ of $P_1$ such that $x(\neq e)\in M_1\cap M_2$. Let $y_1$ and $y_2$ be elements of order $3$ in $P_2$. Suppose $M_1=\{e,x,x_1,x_2\}$ and $M_2=\{e,x,x_3,x_4\}$. Now $x_1,y_1\in \langle x_1y_1\rangle $ and so $xy_1\in \langle x_1y_1\rangle$. Also neither $o(xy_1)\mid o(x_1)$ nor $o(x_1)\mid o(xy_1)$. Consequently, $x_1\sim xy_1$ in $\d$. Similarly, $x_2\sim xy_1$ and $x_3\sim xy_1$ in $\d$. By Proposition \ref{nilpotent coprime adj}, $y_i\sim x_j$ for each $i\in \{1,2\}$ and $j\in \{1,2,3\}$. Thus the subgraph of $\d$ induced by the set $\{y_1,y_2,xy_1,x_1,x_2,x_3\}$ has a subgraph isomorphic to $K_{3,3}$; a contradiction. Hence the intersection of two maximal cyclic subgroups of $P_1$ is trivial.

\textbf{Subcase-1.3:} $P_1\cong \mathbb{Z}_2$. If possible, let $\mathrm{exp}(P_2)>p_2$. Then $P_2$ has at least one cyclic subgroup $H$ of order $p_2^2$. Let $x_1,x_2$ and $x_3$ be generators of $H$ and let $y_1, y_2$ be two elements of order $p_2$ in $H$. Suppose $x\in P_1$ is an element of order $2$. Then $o(xy_1)=o(xy_2)=2p_2$. Also $x,y_1 \in \langle xx_1\rangle$ implies that $xy_1\in \langle xx_1\rangle$ and so $xy_1\sim x_1$ in $\a$. Indeed $xy_i\sim x_j$ in $\a$ for each $i\in \{1,2\}$, $j\in \{1,2,3\}$. But neither $o(xy_i)\mid o(x_j)$ nor $o(x_j)\mid o(xy_i)$ and so $xy_i\sim x_j$ in $\d$ for each $i\in \{1,2\}$, $j\in \{1,2,3\}$. By Proposition \ref{nilpotent coprime adj}, $y\sim x_i$ for each $i\in\{1,2,3\}$. Thus the subgraph induced by set $\{x_1,x_2,x_3,x,xy_1,xy_2\}$ is isomorphic to $K_{3,3}$, a contradiction. Thus, $\mathrm{exp}(P_2)=p_2$.

\textbf{Subcase-1.4:} $P_1\cong \mathbb{Z}_2\times \mathbb{Z}_2\times \cdots \times \mathbb{Z}_2$. In this case, we show that $P_2 \cong \mathbb{Z}_3$. On the contrary, assume that $|P_2|>3$. By Proposition \ref{nilpotent coprime adj}, each non-identity element of $P_1$ is adjacent to every non-identity element of $P_2$, and both $P_1$ and $P_2$ contains at least $3$ non-identity elements. Thus the induced subgraph by the set $(P_1\cup P_2)\setminus \{e\}$ has a subgraph isomorphic to $K_{3,3}$; a contradiction. Thus $P_2 \cong \mathbb{Z}_3$.

\noindent\textbf{Case-2:} $p_1=3$. We discuss this case in the following two subcases.

\textbf{Subcase-2.1:} $P_2 \cong \mathbb{Z}_3$. By replacing $p_1=2$ to $p_1=3$ in the Subcase $1.3$, we obtain $G \cong \mathbb{Z}_3 \times P$, where $P$ is a $p$-group of exponent $p> 3$.

\textbf{Subcase-2.2:} $|P_1|>3$. In this case, $|P_2|>3$. Both $P_1$ and $P_2$ have at least $3$ non-identity elements. Thus, the subgraph induced by the set $(P_1\cup P_2)\setminus \{e\}$ has a subgraph isomorphic to $K_{3,3}$. 

Conversely, let $G \cong \mathbb{Z}_2 \times P$, where $P$ is a $p$-group of exponent $p> 2$. Then by Theorem \ref{star condition}, $\d$ is a star graph, and so $\d$ is planar.

\begin{figure}[h]
    \centering
    \includegraphics[width=1\textwidth]{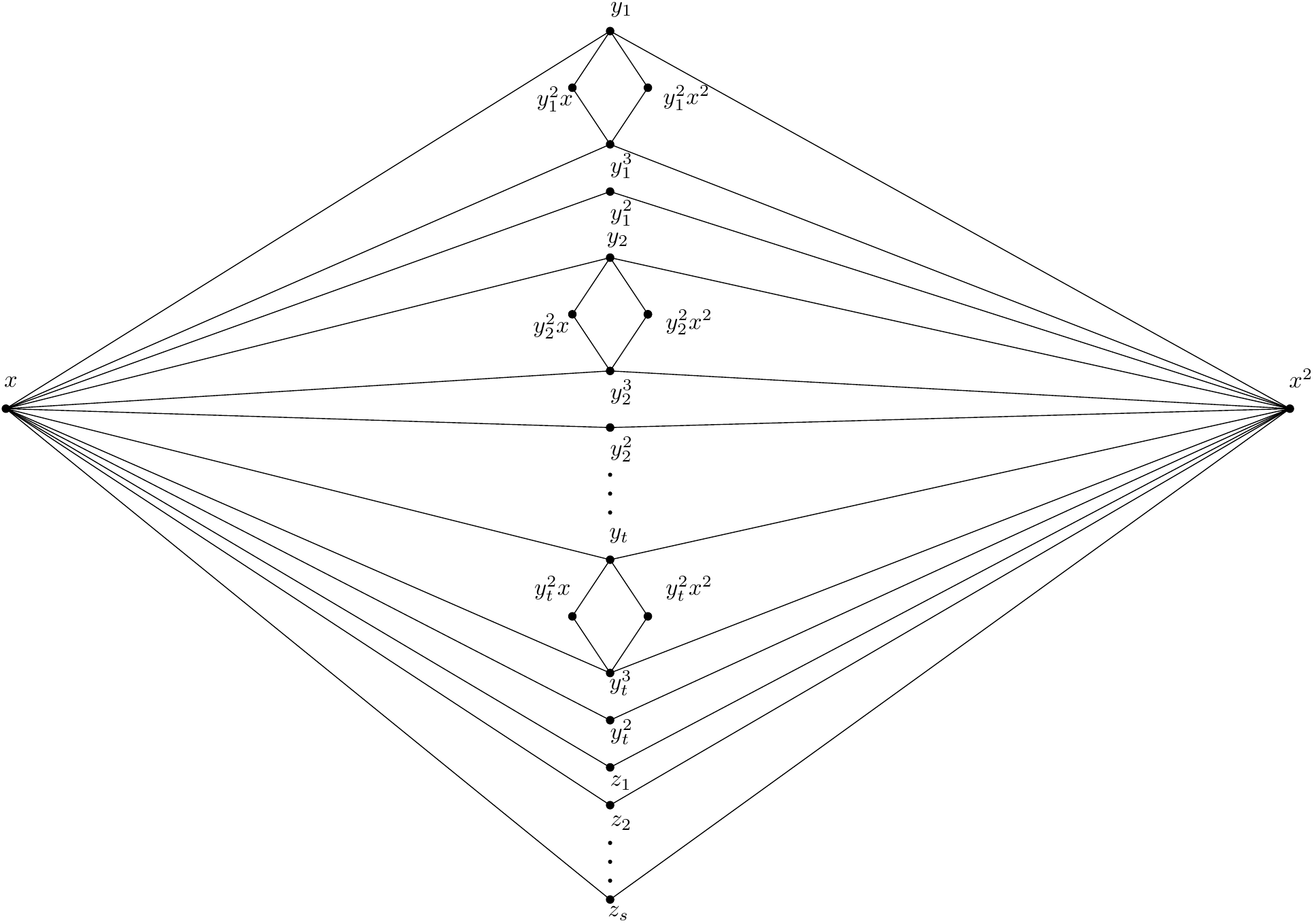}
    \caption{Planar drawing of ${\mathcal{D}(P \times \mathbb{Z}_3)}$.}
    \label{fig1}
\end{figure}

If $G \cong \mathbb{Z}_3 \times P$, where $P$ is a $p$-group of exponent $p> 3$, then $G$ has elements of order $1$, $3$, $p$, and $3p$. Notice that $e\notin V(\d)$ and by Remark $\ref{maximum order}$, the elements of the order $3p$ does not belong to the vertex set of $\d$. By Proposition \ref{nilpotent coprime adj}, each element of order $3$ is adjacent to every element of order $p$ and by Corollary \ref{x nonadj y in P}, $\d$ is a complete bipartite graph $K_{2,|P|-1}$ and so $\d$ is planar. \\
Now let  $G\cong \mathbb{Z}_2 \times \mathbb{Z}_2 \times \cdots \times \mathbb{Z}_2 \times \mathbb{Z}_3$. Then $G$ has elements of order $1$, $2$, $3$ and $6$, respectively. Note that the elements of order $1$ and $6$ do not belong to the vertex set of $\d$. By Proposition \ref{nilpotent coprime adj}, each element of order $2$ is adjacent to every element of order $3$, and by Corollary \ref{x nonadj y in P}, the elements of same order are not adjacent in $\d$. Consequently, $\d$ is a complete bipartite graph $K_{2,t-1}$, where $t=|\mathbb{Z}_2 \times \mathbb{Z}_2 \times \cdots \times \mathbb{Z}_2|$ and so $\d$ is planar.\\
Let $G \cong P \times \mathbb{Z}_3$, where $P$ is a $2$-group with exponent $4$ and any two maximal cyclic subgroups of $P$ has trivial intersection. Suppose $P$ has $t\geq 1$ maximal cyclic subgroups of order $4$ and $s\geq 0$ maximal cyclic subgroups of order $2$. Let the maximal cyclic subgroups of order $4$ are $M_i=\langle y_i\rangle$, where $1\leq i \leq t$. If $s\geq 1$, then $M_j'=\langle z_j\rangle$ where $1\leq j \leq s$ are the maximal cyclic subgroup of order $2$.  Let $H=\langle x\rangle$ be the cyclic subgroup of order $3$ in $G$. Then the maximal cyclic subgroups of order $12$ in $G$ are of the form $M_iH=\{e,x,x^2,y_i,y_i^2,y_i^3,y_ix,y_i^2x,y_i^3x,y_ix^2,y_i^2x^2,y_i^3x^2\}$. Also, the maximal cyclic subgroups of order $6$ are of the form $M_j'H=\{e,x,x^2,z_j, z_jx,z_jx^2\}$ (see {\rm \cite[Lemma 2.11]{a.chattopadhyay2021minimal}}). Note that the identity element and the generators of these maximal cyclic subgroups do not belong to the vertex set of $\d$. A planar drawing of $\d$ is given in Figure \ref{fig1}. Thus the graph $\d$ is planar.
\end{proof}

\begin{corollary}
      Let $G$ be a finite cyclic group which is not a $p$-group. Then the difference graph $\d$ is  planar if and only if $G$ is isomorphic to one of the groups:
      $\mathbb{Z}_{12}$, $\mathbb{Z}_{2p}$ or $\mathbb{Z}_{3q}$,  where $p>2$ and $q>3$ are primes.
\end{corollary}

\begin{theorem}
Let $G$ be a finite nilpotent group which is not a $p$-group. Then $\d$ is an outerplanar graph if and only if $G \cong \mathbb{Z}_2 \times P$, where $P$ is a $p$-group of exponent $p> 2$.
\end{theorem}
\begin{proof}
We first assume that $\d$ is an outerplanar graph. It is well-known that every outerplanar graph is a planer graph. In view of Theorem \ref{PlanarD}, if $G \cong \mathbb{Z}_3 \times P$, where $P$ is a $p$-group of exponent $p> 3$, then by Proposition \ref{nilpotent coprime adj}, each non-identity element of $\mathbb{Z}_3$ is adjacent to every non-identity element of $P$. Since $|P|\geq 5$, we get a subgraph of $\d$ which is isomorphic to $K_{2,3}$, a contradiction. Now let $G \cong P \times \mathbb{Z}_3$, where $P$ is a $2$-group with exponent $4$ and any two maximal cyclic subgroups of $P$ has the trivial intersection. Then from Figure \ref{fig1}, observe that the subgraph induced by the set $\{x, x^2,y_1,y_1^2,y_1^3\}$ is isomorphic to $K_{2,3}$, which is not possible. If $G\cong  P \times \mathbb{Z}_3$, where $P=\mathbb{Z}_2 \times \mathbb{Z}_2 \times \cdots \times \mathbb{Z}_2$, then by Proposition \ref{nilpotent coprime adj}, the subgraph of $\d$ induced by the set $(P\cup \mathbb{Z}_3)\setminus \{e\}$ has a subgraph isomorphic to $K_{2,3}$. Consequently $G \cong \mathbb{Z}_2 \times P$, where $P$ is a $p$-group of exponent $p> 2$. 

 Conversely, if $G \cong \mathbb{Z}_2 \times P$, where $P$ is a $p$-group of exponent $p> 2$, then by Theorem \ref{star condition}, $\d$ is a star graph. Hence, $\d$ is an outerplanar graph.
\end{proof}


\section{Difference graph of non-nilpotent groups}

In this section, we investigate the difference graphs of non-nilpotent groups $S_n$ and $A_n$ with some forbidden induced subgraphs. Note that for $1\leq n \leq 4$ and $1\leq n \leq 6$, $S_n$  and $A_n$, respectively, are EPPO-groups. Hence, their difference graphs are null graphs (i.e., the vertex set and edge sets are empty) for the corresponding values of $n$. In the following, we thus consider $\mathcal{D}(S_n)$  and $\mathcal{D}(A_n)$ for $n \geq 5$ and $n \geq 7$, respectively.


In the following, we determine the values of $n$ for which $\mathcal{D}(S_n)$ is cograph, chordal, split, and threshold. 

\begin{theorem}
For $n\geq 5$, $\mathcal{D}(S_n)$ is a chordal graph if and only if $n=5$.
\end{theorem}

\begin{proof} First, we show that for $n\geq 6$, $\mathcal{D}(S_n)$ has an induced cycle isomorphic to $C_4$. Notice that $H_1=\langle (1 2), (3 4 5 )\rangle$, $H_2=\langle (3 4 5 ),(1 6)\rangle $, $H_3=\langle (1 6),(354)\rangle$ and  $H_4=\langle (3 5 4),(1 2)\rangle$ are cyclic subgroups of $S_6$. Thus, $(1 2)\sim (3 4 5)\sim (1 6)\sim (3 5 4)\sim (1 2)$ in $\mathcal{P}_E(S_6)$. By Remark \ref{order not divide}, $(1 2)\nsim (3 4 5)\nsim (1 6)\nsim (3 5 4)\nsim (1 2)$ in $\mathcal{P}(S_6)$. By Proposition \ref{order divide}, $(1 2)\nsim (1 6)$ and $(3 4 5)\nsim (3 5 4)$ in $\mathcal{D}(S_6)$.
Consequently, the subgraph induced by these four vertices in $\mathcal{D}(S_6)$ is isomorphic to $C_4$. Hence, $\mathcal{D}(S_n)$ is not a chordal graph for $n\geq 6$.

 Now we shall  that $\mathcal{D}(S_5)$ is a chordal graph. On the contrary, assume that $\mathcal{D}(S_5)$ has an induced subgraph isomorphic to a cycle graph of at least $4$ vertices say $x\sim y \sim z \sim \cdots \sim t\sim x$. Observe that $\pi _{S_5}=\{1,2,3,4,5,6\}$. Also note that each element of order $4$, $5$, $6$ generates a maximal cyclic subgroup of $S_5$. Consequently,  $o(x),o(y),o(z),o(t)\in \{2,3\}$. Notice that each element of order $6$ in $S_5$ is of the form $(abc)(de)$ and each element of order $3$ commutes with exactly one element of order $2$. Now, let $o(x)=2$. Then by Proposition \ref{order divide}, $o(y)=3$ and $o(z)=2$. Since $x\sim y$ and $y\sim z$ in $\mathcal{D}(S_5)$, we get $\langle x,y\rangle$ and $\langle y,z\rangle$ are cyclic subgroups of $S_5$. It follows that $xy=yx$ and $yz=zy$. Consequently, $x=z$, which is not possible.
 Similarly, if $o(x)=3$, then we obtain $y=t$, which is not possible. Hence, $\mathcal{D}(S_5)$ is a chordal graph.
\end{proof}

\begin{theorem}
For $n\geq 5$, $\mathcal{D}(S_n)$ is a cograph if and only if $n=5$.
\end{theorem}

\begin{proof}
For $n=6$, notice that the induced path $(15)(2 4)(3 6)\sim (1 6 4) (2 5 3)\sim (1 2)(3 4)(5 6) \sim (1 5 3)(4 2 6)$ in $\mathcal{P}_E(S_6)$ is isomorphic to $P_4$. By Remark \ref{order not divide}, $(15)(2 4)(3 6)\nsim (1 6 4) (2 5 3)$, $(1 6 4) (2 5 3)\nsim (1 2)(3 4)(5 6)$ and  $(1 2)(3 4)(5 6) \nsim (1 5 3)(4 2 6)$ in $\mathcal{P}(S_6)$. We obtain an induced path isomorphic to $P_4$ in $\mathcal{D}(S_6)$. Note that $S_n(n\geq 6)$ has a subgroup isomorphic to $S_6$. By Lemma \ref{induced lemma}, for $n\geq 6$, $\mathcal{D}(S_n)$ is not a cograph.\\
Conversely, if $n=5$, then we show that $\mathcal{D}(S_5)$ is a cograph. On the contrary, assume that $\mathcal{D}(S_5)$ has an induced path $P: x\sim y \sim z \sim t$ isomorphic to $P_4$. Notice that $\pi _{S_5}=\{1,2,3,4,5,6\}$ and each element of the order $4$, $5$, $6$ generates a maximal cyclic subgroup of $S_5$. Consequently, elements of these orders do not belong to the vertex set of $\mathcal{D}(S_5)$ (cf. Proposition \ref{proposition 2.1}). It follows that the order of each vertex in $P$ is either $2$ or $3$. Let $o(x)=2$. Then by Proposition \ref{order divide}, $o(y)=o(t)=3$ and $o(z)=2$. Note that $\langle x,y\rangle$ is a cyclic subgroup of order $6$. Without loss of generality, assume that $x=(1 2)$ and $y=(3 4 5 )$. Also $\langle y,z\rangle $ is a cyclic subgroup of order $6$. It follows that $z=( 1 2 )$, which is not possible.\\
Now let $o(x)=3$. Then by Proposition \ref{order divide}, $o(y)=o(t)=2$ and $o(z)=3$. Observe that $\langle x,y\rangle$ is a cyclic subgroup of order $6$ in $S_5$. Without loss of generality, suppose $x=(1 2 3)$ and $y=(4 5 )$. Note that $\langle y,z\rangle $ is a cyclic subgroup of order $6$. It follows that either $z=( 1 3 2)$ or $z=(1 2 3)$. But $x\neq z$ implies  $z=(1 3 2)$. Since $\langle z,t\rangle $ is a cyclic subgroup of order $6$, we obtain $t=(4 5)$, which is not possible. Hence, $\mathcal{D}(S_5)$ is a cograph.
\end{proof}

\begin{theorem}
For $n\geq 5$, $\mathcal{D}(S_n)$ is neither a split graph nor a threshold graph.
\end{theorem}

\begin{proof}
 Notice that $H_1= \langle (1 2),(3 4 5)\rangle$ and $H_2= \langle ( 1 3),(2 4 5)\rangle$ are cyclic subgroups of order $6$ of $S_5$. Thus, $(1 2)\sim ( 3 4 5)$ and $(1 3)\sim (2 4 5)$ in $\mathcal{P}_E(S_5)$. By Remark \ref{order not divide}, $(1 2)\nsim ( 3 4 5)$ and $(1 3)\nsim (2 4 5)$ in $\mathcal{P}(S_5)$. Consequently, $(1 2)\sim ( 3 4 5)$ and $(1 3)\sim (2 4 5)$ in $\mathcal{D}(S_5)$. Now, by Proposition \ref{order divide}, $(1 2)\nsim (1 3)$ and $(3 4 5)\nsim (2 4 5)$ in $\mathcal{D}(S_5)$. Also, $\langle ( 1 2),(2 4 5)\rangle$ and $\langle ( 1 3),(3 4 5)\rangle$ are non-cyclic subgroups of $S_5$. Consequently, $(1 2)\nsim (2 4 5)$ and $(1 3)\nsim (3 4 5)$ in $\mathcal{D}(S_5)$. Thus, the subgraph induced by the set $\{(1 2),(3 4 5),( 1 3),(2 4 5)\}$ is isomorphic to $2K_2$ in $\mathcal{D}(S_5)$. Hence, $S_n(n\geq 5)$ can neither be a split graph nor a threshold graph.
\end{proof}

We next study the difference graph of alternating groups.

\begin{theorem}
For $n\geq 7$, $\mathcal{D}(A_n)$ is not a chordal graph.
\end{theorem}
\begin{proof}
 Notice that $(12)(34)\sim (567)\sim (13)(24)\sim (576)\sim (12)(34)$ is an induced cycle of $\mathcal{D}(A_7)$ which is isomorphic to $C_4$. It follows that $\mathcal{D}(A_7)$ is not a chordal graph. Hence, $\mathcal{D}(A_n)(n\geq 7)$ can not be a chordal graph.
\end{proof}
By the proof of the above theorem, observe that $\mathcal{D}(A_n)(n\geq 7)$ contains an induced cycle isomorphic to $C_4$. Thus, we have the following corollary.
\begin{corollary}
For $n\geq 7$, $\mathcal{D}(A_n)$ is neither a split graph nor a threshold graph.
\end{corollary}

\begin{theorem}
For $n\geq 7$, $\mathcal{D}(A_n)$ is a cograph if and only if $n=7$.
\end{theorem}
\begin{proof}
 Note that $(12)(34)\sim (567)\sim (12)(38)\sim (564)$ is an induced path in $\mathcal{D}(A_8)$ which is isomorphic to $P_4$. Thus, $\mathcal{D}(A_8)$ is not a cograph.  It follows that for $n\geq 8$, $\mathcal{D}(A_n)$ is not a cograph. Now, we show that $\mathcal{D}(A_7)$ is a cograph. On contrary, assume that $\mathcal{D}(A_7)$ has an induced path $P:x\sim y\sim z \sim t$ isomorphic to $P_4$. Note that $\pi_{A_7}=\{1,2,3,4,5,6,7\}$. Also, each element of the order $4$,$5$,$6$,$7$ generates a maximal cyclic subgroup of $A_7$. Consequently, the elements of these order do not belong to the vertex set of $\mathcal{D}(A_7)$. It follows that $o(x),o(y),o(z),o(t)\in \{2,3\}$. Let $o(x)=2$. Then $o(y)=o(t)=3, o(z)=2$. Observe that $\langle x,y\rangle$ is a cyclic subgroup of order $6$ and each element of order $6$ in $A_7$ has the cycle decomposition $\{2,2,3\}$. Without loss of generality, assume that $x=(12)(34)$ and $y=(567)$. Since $\langle y, z\rangle$ and $\langle z,t \rangle$ are cyclic subgroups of order $6$, we get either $z=(13)(24)$ or $(14)(23)$. It follows that $t=(576)$. In this case, $t \sim x$ in $\mathcal{D}(A_7)$, a contradiction.\\ 
Now let $o(x)=3$. Then $o(y)=o(t)=2$ and $o(z)=3$. Note that $\langle x,y\rangle$ is a cyclic subgroup of order $6$ and each element of order $6$ in $A_7$ has the cycle decomposition $\{2,2,3\}$. Without loss of generality, assume that $x=(123)$ and $y=(45)(67)$. Since $\langle y, z\rangle$ and $\langle z,t \rangle$ are cyclic subgroups of order $6$, we get $z=(132)$. Consequently, either $t=(46)(57)$ or $(47)(65)$. In both of these cases, $t \sim x$ in $\mathcal{D}(A_7)$, is again a contradiction. Thus, $\mathcal{D}(A_7)$ is a cograph.
\end{proof}
\section{Conclusion and open problems}
The study of difference graph $\d$ of a finite group $G$ was initiated in \cite{a.biswas2022difference}. In this article, we continue the study of difference graph $\d$ of a finite group $G$ and classify all the nilpotent groups $G$ such that $\d$ is a chordal graph, cograph, threshold graph, bipartite, Eulerian, planar, and outerplanar. Moreover, for a nilpotent group $G$
. Also, in this article, we study the difference graph of some non-nilpotent groups. In this connection, we characterize all the values of $n$ such that the difference graph of the symmetric group $S_n$ (or alternating group $A_n$) is cograph and chordal. However, some of the results which answers to the following natural questions will appear in continuation of this manuscript. 

\vspace{.2cm}
\textbf{Question 1.}
Classification of finite groups such that the crosscap of $\d$ is at most $2$.

\vspace{.2cm}
\textbf{Question 2.}
Classification of finite groups such that the  genus of $\d$ is at most $2$.

\section{Declarations}
	
\textbf{Funding:} The first author gratefully acknowledge for providing financial support to CSIR  (09/719(0110)/2019-EMR-I) government of India. The second author wishes to  acknowledge the support of Core Research Grant (CRG/2022/001142) funded by SERB, government of India. 

\vspace{.3cm}
 \textbf{Data availability:}  Data sharing not applicable to this article as no datasets were generated or analysed during the current study.

\vspace{.3cm}
 \textbf{Conflict of interest:} On behalf of all authors, the corresponding author states that there is no conflict of interest.

\vspace{1cm}
\noindent
{\bf Parveen\textsuperscript{\normalfont 1}, \bf Jitender Kumar\textsuperscript{\normalfont 1}, \bf Ramesh Prasad Panda\textsuperscript{\normalfont 2}}
\bigskip

\noindent{\bf Addresses}:

\vspace{5pt}

\end{document}